\documentclass[11pt]{article}
\usepackage{multicol,amsmath,inputenc,amsfonts,amsthm,pgfplots,tikz,fancyhdr,enumitem,graphicx, mathrsfs, bm, tikz-cd, amssymb, hyperref}
\usepackage[center]{titlesec}
\usetikzlibrary{arrows}
\allowdisplaybreaks

\usepackage[margin=1in]{geometry}
\pgfplotsset{width=15cm,compat=1.9}

\usepackage[none]{hyphenat}
\newtheorem{theorem}{Theorem}
\newtheorem{proposition}{Proposition}
\newtheorem*{definition}{Definition}
\newtheorem*{acknowledgments}{Acknowledgments}
\newtheorem*{remark}{Remark}
\newtheorem{corollary}{Corollary}[theorem]
\newtheorem{lemma}[theorem]{Lemma}
\newcommand{\Q}{\mathbb{Q}}

\newcommand{\R}{\mathbb{R}}
\newcommand{\Z}{\mathbb{Z}}
\newcommand{\N}{\mathbb{N}}

\newcommand{\Ln}{\mathcal{L}}

\renewcommand{\vec}[1]{\mathbf{#1}}

\title{\textbf{A Khintchine type theorem for affine subspaces}}
\author{Daniel C. Alvey\footnote{\textbf{email:} \texttt{dalvey@wesleyan.edu}} \\ Wesleyan University}
\date{}

\begin{document}
\maketitle

\pagestyle{fancy}\renewcommand{\headrulewidth}{0pt}\fancyhf{}
\lhead{}
\cfoot{\thepage}

\begin{abstract}
\noindent We show that affine subspaces of Euclidean space are of Khintchine type for divergence under certain multiplicative Diophantine conditions on the parametrizing matrix. This provides evidence towards the conjecture that all affine subspaces of Euclidean space are of Khintchine type for divergence, or that Khintchine's theorem still holds when restricted to the subspace. This result is proved as a special case of a more general Hausdorff measure result from which the Hausdorff dimension of $W(\tau)$ intersected with an appropriate subspace is also obtained.
\end{abstract}

\section{Introduction}

Simultaneous Diophantine approximation is focused on how well points in $\R^d$ can be approximated by points in $\Q^d$. Dirichlet's simultaneous approximation theorem tells us that for every point $\vec{x} \in \R^d,$ we have that $\lVert q\vec{x}\rVert< q^{-1/d}$ for infinitely many $q \in \Z$ where $\lVert \cdot \rVert$ is the sup norm distance to the nearest integer point. A natural question is whether this bound can be improved, and it was shown by Khintchine \cite{khintchine1926} that if we consider the $\psi$-\textit{approximable} $\vec{x} \in \R^d$, for which $\lVert q\vec{x}\rVert<\psi(q)$ for infinitely many $q \in \N$ where $\psi: \N \to \R^+$ is any decreasing function, almost every or almost no points are $\psi$-approximable depending on if $\sum_{q \in \N} \psi(q)^d$ diverges or converges, respectively. 

However, Khintchine's theorem tells us nothing about the size of the $\psi$-approximable points in subspaces of measure zero. It would be desirable to obtain a statement such as the following: The measure of the $\psi$-approximable points in a subspace $\mathcal{M} \subseteq \R^d$ of dimension $n$ is of zero or full $n$-dimensional measure depending on if the sum $\sum_{q \in \N} \psi(q)^d$ converges or diverges. However, if $\sum_{q \in \N}\psi(q)^2$ converges there are explicit examples of lines in the plane for which the above statement is not true. As such, it makes sense to focus efforts on each side of the problem separately. To do so, the concept of \textit{Khintchine type} is very helpful.

A subspace $\mathcal{M} \subseteq \R^d$ is of \textit{Khintchine type for divergence} if for any decreasing function $\psi: \N \to \R^+$ such that $\sum_{q \in \N}\psi(q)^d$ diverges, almost every point on $\mathcal{M}$ is $\psi$-approximable, and of \textit{Khintchine type for convergence} if almost no point is $\psi$-approximable when the sum converges. Any analytic, non-degenerate manifold of Euclidean space is of Khintchine type for divergence \cite{beresnevich2012}, and a large class of manifolds are of Khintchine type for convergence \cite{BRVZ17, simmons2018}.

However, less is known about affine subspaces. Currently, from \cite{RSS17} it is known that affine coordinate subspaces of dimension at least two, and affine coordinate subspaces of dimension one with a Diophantine restriction on the shift vector are of Khintchine type for divergence. In the convergence side of the theory, certain coordinate hyperplanes are of Khintchine type for convergence \cite{R15} and affine subspaces with a Diophantine restriction on the parametrizing matrix are of Khintchine type for convergence \cite{huang2018}.

It is worth noting that these same questions may be asked in the case of \textit{dual approximation} where the interest is in measure theoretic properties of the set $\vec{x} \in \R^d$ such that $\lVert \langle \vec{x}, \vec{q} \rangle \rVert<\psi\left(|\vec{q}|\right)$ for infinitely many $\vec{q} \in \Z^d$. In this dual case, there is a similar concept of dual Khintchine type for divergence and convergence. It is known that any analytic, non-degenerate manifold is of both dual Khintchine type for convergence  \cite{BKM01} and divergence  \cite{BBKM02}. 

As in the simultaneous case, less is known about affine subspaces for dual approximation. Lines through the origin with a Diophantine restriction are of dual Khintchine type for convergence  \cite{BBDD00}, and hyperplanes with a Diophantine restriction are of both dual Khintchine type for convergence \cite{ghosh05}, and divergence \cite{ghosh11}.

The purpose of this paper is to prove a simultaneous, divergence Khintchine type theorem for affine subspaces with a certain Diophantine restriction on the choice of parametrizing matrix. First we need to introduce some notation. After a change of variables, any $n$-dimensional affine subspace can be defined as 
$$\mathcal{L}=\{(\vec{x},\vec{x}\mathfrak{a}+\vec{a}_0) \in \R^d : \vec{x} \in \R^{n}\} \quad \text{or} \quad \mathcal{L}=\{(\vec{x},\tilde{\vec{x}}\tilde{\mathfrak{a}}) \in \R^d : \vec{x} \in \R^{n}\}$$
where $\vec{x}=(x_1, \ldots, x_{n})\in \R^{n}$, $\mathfrak{a} \in \text{Mat}_{n,d-n}$, $\vec{a}_0 \in \R^{d-n}$, $\tilde{\vec{x}}=(1,x_1, \ldots, x_{n}) \in \R^{n+1}$, and $\tilde{\mathfrak{a}}\in \mbox{Mat}_{n+1,d-n}$ such that $\tilde{\mathfrak{a}}=\left[\begin{array}{l}\vec{a}_0 \\\mathfrak{a} \end{array}\right]$. Note that the matrix $\mathfrak{a}$ accounts for the ``tilt" of the subspace, and the vector $\vec{a}_0$ accounts for the ``shift" of the subspace off the origin.

Also necessary for this result is the concept of multiplicative Diophantine approximability of matrices. For $\omega>0,$ define
\begin{equation*}
\textbf{MAD}(m,n,\omega)=\left\lbrace \mathfrak{a} \in \text{Mat}_{m,n} : \inf_{\vec{j} \in \Z^n  \setminus \{\vec{0}\}} |\vec{j}|^{\omega} \prod_{u=1}^m \lVert \vec{j} \cdot \text{row}_u(\mathfrak{a}) \rVert >0 \right\rbrace.
\end{equation*}
We say that $\mathfrak{a} \in \text{Mat}_{m,n}$ has multiplicative Diophantine exponent $\omega(\mathfrak{a})$ if 
$$\omega(\mathfrak{a})=\inf \left\lbrace \omega \in \R: \mathfrak{a} \in \textbf{MAD}(m,n,\omega) \right\rbrace.$$
If $\mathfrak{a} \not\in \textbf{MAD}(m,n,\omega)$ for any $\omega \in \R$ then we say that $\omega (\mathfrak{a})=\infty.$
A matrix $\mathfrak{a} \in \text{Mat}_{m,n}$ is said to be \textit{multiplicatively badly approximable} if $\mathfrak{a} \in \textbf{MAD}(m,n,1).$ It is worth noting that in the case where $m=n=1$, the set $\textbf{MAD}(1,1,1)$ is exactly equal to the set of badly approximable numbers. Additionally, the Littlewood conjecture can be restated as  asserting that $\textbf{MAD}(2,1,1)$ is empty, and if $\mathfrak{a}\in \text{Mat}_{1,n}$, then $\omega(\mathfrak{a})$ is the dual Diophantine type $\mathfrak{a}$ when thought of as a vector. The reader is encouraged to refer to \cite{bugeaud09} for a survey of this concept. With this notation, We will show the following theorems.
\begin{theorem}\label{main theorem}
 Let $\mathcal{L}=\{(\vec{x},\vec{x}\mathfrak{a}+\vec{a}_0) \in \R^d : \vec{x} \in \R^{n}\}$ be an $n$-dimensional affine subspace of $\R^d$ with  $\mathfrak{a}\in \textup{Mat}_{n,d-n}$ and $\vec{a}_0 \in \R^{d-n}.$ Then $\mathcal{L}$ is of Khintchine type for divergence if $\omega(\mathfrak{a})<dn$.

\end{theorem}

This requirement on $\mathfrak{a}$ is the exact same as the requirement for an affine subspace to be of strong Khintchine type for convergence in  \cite[Theorem 1]{huang2018}. Note that this restriction is only on the part of the parametrizing matrix responsible for the ``tilt" of the subspace. Therefore, whether or not a given affine subspace with the ``correct" tilt is of Khintchine type for divergence is independent of its shift off the origin. 

It is interesting to note that \cite[Theorem 2,(ii)]{RSS17} is a result strictly about affine coordinate subspaces, which states that an affine coordinate subspace $\{\vec{x}\} \times \R$ of $\R^d$ of dimension one is of Khintchine type for divergence if $\omega(\vec{x})<d$. Since the dimension of this subspace is one, this restriction also aligns exactly with the restriction in Theorem \ref{main theorem}. This is of note since the restriction of Theorem \ref{main theorem} is only on the tilt of the subspace, while not making any statement about subspaces which are not tilted. Whereas, the restriction of \cite[Theorem 2,(ii)]{RSS17} is only on the shift of the subspace, while not making any statement about subspaces which are tilted in any way. This is made more interesting by the fact that the methods of proof also are different. While both utilize the ubiquitous systems framework, the counting result is shown by using a geometry of numbers argument in \cite[Theorem 2,(ii)]{RSS17} and by using Selberg functions to translate the counting problem into a dual problem in this result.

Theorem \ref{main theorem} follows as a special case of the following Jarn\'ik type for divergence theorem regarding the Hausdorff measure of the set $W(\psi) \cap \Ln$.

\begin{theorem}\label{main jarnik}
Let $\mathcal{L}=\{(\vec{x},\vec{x}\mathfrak{a}+\vec{a}_0) \in \R^d : \vec{x} \in \R^{n}\}$ be an $n$-dimensional affine subspace of $\R^d$ with  $\mathfrak{a}\in \textup{Mat}_{n,d-n}$ and $\vec{a}_0 \in \R^{d-n}$ such that $\omega(\mathfrak{a})<\frac{n(d-n+s)}{n+1-s}.$ Let $0\leq s\leq n$ and let $\psi: \R \to \R^+$ be a decreasing function. Then 
\begin{equation}\label{jarnik}
\mathcal{H}^s(W(\psi) \cap \Ln)=\mathcal{H}^s(\Ln) \quad \textup{if} \quad \sum_{q=1}^{\infty} \psi(q)^{d-n+s}q^{n-s} = \infty.
\end{equation}
\end{theorem}

Note that if $s<n$, then this implies that $\mathcal{H}^s(W(\psi) \cap \Ln)=\infty$. This requirement on $\mathfrak{a}$ is the exact requirement for an affine subspace to be of strong Jarn\'ik type for convergence in \cite[Theorem 2]{huang2018}. Note that in the case where $s=n$, $\mathcal{H}^n$ is comparable to Lebesgue measure on $\R^n$, so Theorem \ref{main theorem} can be thought of as the special case of Theorem \ref{main jarnik} when $s=n.$ Additionally, this result allows us to obtain a lower bound for the Hausdorff dimension of the important set $W(\tau)\cap\Ln.$ Define
$$W(\tau)=\{ \vec{x} \in \R^n : \lVert q \vec{x} \rVert < q^{-\tau} \text{ for infinitely many } q \in \N\},$$
and note that $W(\tau)=W(\psi)$ where we choose $\psi(q)=q^{-\tau}.$ In $\R$, the famous Jarn\'ik-Besicovitch theorem states that the Hausdorff dimension of this set is $\frac{2}{\tau+1}$ for $\tau>1$. It is natural, then, to ask about the Hausdorff dimension of $W(\tau)$ on subspaces of Euclidean space. There are a constellation of results regarding this question, known as the Generalized Baker-Schmidt Problem, a good summary of which can be found in the introduction of \cite{GBSP}. Choosing $\psi(q)=q^{-\tau}$ in Theorem \ref{main jarnik}, we obtain a lower bound on $\dim W(\tau) \cap \Ln$. An upper bound on $\dim W(\tau) \cap \Ln$ can be obtained from Theorem 2 in \cite{huang2018}, which together immediately gives the following corollary.

\begin{corollary} Let $\mathcal{L}=\{(\vec{x},\vec{x}\mathfrak{a}+\vec{a}_0) \in \R^d : \vec{x} \in \R^{n}\}$ be an $n$-dimensional affine subspace of $\R^d$ with  $\mathfrak{a}\in \textup{Mat}_{n,d-n}$ and $\vec{a}_0 \in \R^{d-n}$ such that $\omega(\mathfrak{a})<\frac{n}{\tau}$ for some $\tau \geq \frac{1}{d}.$ Then
$$\dim W(\tau) \cap \Ln = n-\dfrac{\tau d-1}{\tau+1}.$$
\end{corollary}
It is worth noting that the upper bound can also be established with the less restrictive condition that $\omega(\tilde{\mathfrak{a}})< \frac{n+1}{\tau}$, and is found as Corollary 4 in \cite{huang2018}. The upper bound used above is derived from Theorem 2 in \cite{huang2018}, which in fact gives a condition for an affine subspace to be of strong Jarn\'ik type for convergence. The author would like to thank Jing-Jing Huang for drawing his attention to this aspect of the problem.

\section{Indexing of rational points}\label{index}

Since this problem is concerned with the extrinsic Diophantine approximation of subspaces, it is important to find a way of identifying those rational points $\vec{m}/q \in \Q^d$ which are suitably close to the subspace in question.. Additionally, it will be useful to focus attention on segments of $\Ln$ which intersect with a strip of $\R^d.$

For $\vec{v} \in \Z^n$ define $\mathcal{S}_\vec{v}=\prod_{i=1}^n[v_i,v_i+1]^n \times \R^{d-n}$. Note that then $\mathcal{S}_{\vec{0}} = [0,1]^n \times \R^{d-n}$, and 
$$\bigcup_{\vec{v} \in \Z^n} \mathcal{S}_{\vec{v}}= \R^d.$$

 We will consider the problem of showing that $\mathcal{H}^s\left(W(\psi) \cap \left(\Ln \cap \mathcal{S}_{\vec{0}}\right)\right)=\mathcal{H}^s(\Ln\cap \mathcal{S}_\vec{0})$.  Doing so will allow us to consider the problem on a compact metric space, $[0,1]^n$, which is necessary to use the tools of ubiquity theory. This focus on the unit $n$-cube will be justified in the proof of Theorem 1.

For each $q \in \N$, consider those rational points $\vec{m}/q \in \Q^d$ which are in a suitable neighborhood of $\mathcal{L}_{\tilde{\mathfrak{a}}}^n$. In other words, rational points such that
$$\left|\vec{y}-\dfrac{\vec{m}}{q}\right|<\dfrac{\psi(q)}{q}$$
for some $\vec{y} \in \mathcal{L}_{\tilde{\mathfrak{a}}}^n$ where $| \cdot |$ denotes the sup norm. Put differently, there exists a rational point $\vec{m}/q \in \Q^d$ $\psi$-close to some $\vec{x}\in \Ln$ if
$$\lVert q\vec{y}\rVert<\psi(q).$$

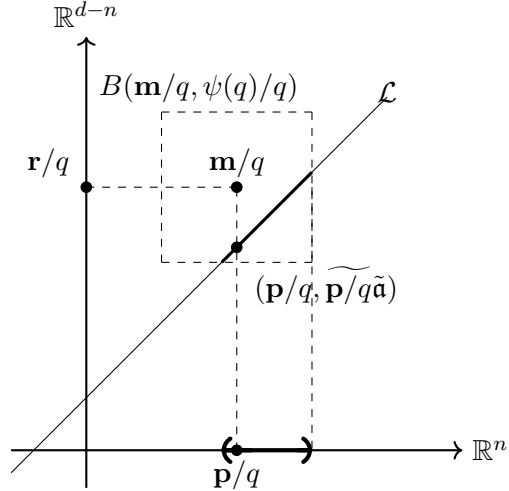
\begin{figure}[h]
\begin{center}
\begin{tikzpicture}[scale=1]
\draw [black] (-1,-.8) -- (0,0.2) -- (4,4.2);
\node [above] at (4,4) {$\Ln$};
\draw [dashed] (2,-.5) -- (2,3);
\node [below] at (2,-.5) {$\vec{p}/q$};
\filldraw [black] (2,2.2) circle (2pt);
\node [right] at (2.1,1.7) {$(\vec{p}/q,\widetilde{\vec{p}/q}\tilde{\mathfrak{a}})$};
\filldraw [black] (2,3) circle (2pt);
\node [above] at (2,3) {$\vec{m}/q$};
\draw [dashed] (0,3) -- (2,3);
\filldraw [black] (0,3) circle (2pt);
\filldraw [black] (2,-.5) circle (2pt);
\node [above] at (-.5,3) {$\vec{r}/q$};
\node [above] at (1.5,4) {$B(\vec{m}/q,\psi(q)/q)$};
\draw [dashed] (1,4) -- (1,2);
\draw [dashed] (3,4) -- (3,2);
\draw [dashed] (1,2) -- (3,2);
\draw [dashed] (1,4) -- (3,4);
\draw [very thick] (1.8,2) -- (3,3.2);
\draw [dashed] (3,3) -- (3,-.5);
\draw [ultra thick] (1.8,-.5)--(3,-.5);
\draw [ultra thick, -)] (2.8,-.5)--(3,-.5);
\draw [ultra thick, (-] (1.8,-.5)--(2.2,-.5);
\draw [thick, ->] (-1,-.5) -- (0,-.5) -- (5,-.5);
\draw [thick, ->] (0,-1) -- (0,0) -- (0,5);
\node [above] at (0,5) {$\R^{d-n}$};
\node [right] at (5,-.5) {$\R^n$};
\end{tikzpicture}
\caption{Possible off-center neighborhoods}
\end{center}
\end{figure}

To index these points, consider a rational point $\vec{p}/q \in \Q^{n} \cap [0,1]^n$ and the point $(\vec{p}/q,\widetilde{\vec{p}/q}\tilde{\mathfrak{a}})$ on $\mathcal{L}_{\tilde{\mathfrak{a}}}^n$.  If there exists an $\vec{r}/q \in \Q^{d-n}$ such that 
$$\left|\widetilde{\vec{p}/q}\tilde{\mathfrak{a}}-\dfrac{\vec{r}}{q}\right|< \dfrac{\psi(q)}{q}$$
then the above inequality can be rewritten as

\begin{equation}
\lVert\hat{\vec{p}} \tilde{\mathfrak{a}} \rVert< \psi \left( \left| \hat{\vec{p}}\right| \right)
\end{equation}

where $\hat{\vec{p}}=(q,\vec{p}) \in \N^{n+1}$ noting that since $\vec{p}/q \in [0,1]^n$, $|\hat{\vec{p}}|=q$. This inequality is a nice shorthand way of saying that the point $\vec{m}/q=(\vec{p}/q,\vec{r}/q) \in \Q^d$ is $\psi$-close to $\Ln$ if the corresponding vector $\hat{\vec{p}} \in \N^{n+1}$ satisfies the above inequality.

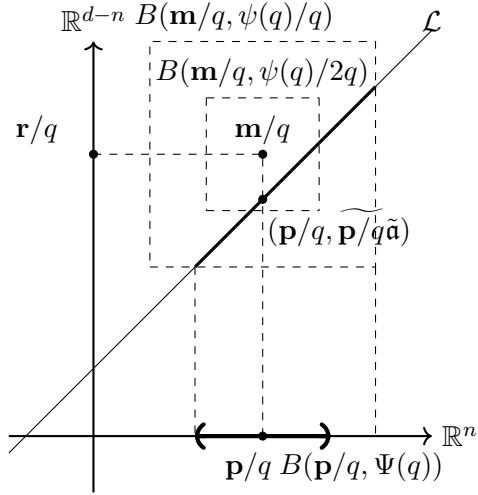
\begin{figure}[h]
\begin{center}
\begin{tikzpicture}[scale=.75]
\draw [black] (-2.5,-2.3) -- (0,.2) -- (5,5.2);
\node [above] at (5,5) {$\Ln$};
\draw [dashed] (2,-2) -- (2,3);
\filldraw [black] (2,-2) circle (2pt);
\node [below] at (1.75,-2.1) {$\vec{p}/q$};
\node [below] at (3.75,-2.1){$B(\vec{p}/q, \Psi(q))$};
\filldraw [black] (2,2.2) circle (2pt);
\node [right] at (1.9,1.7) {$(\vec{p}/q,\widetilde{\vec{p}/q}\tilde{\mathfrak{a}})$};\filldraw [black] (2,3) circle (2pt);
\node [above] at (2,3) {$\vec{m}/q$};
\draw [dashed] (-1,3) -- (2,3);
\filldraw [black] (-1,3) circle (2pt);
\node [above] at (-2,3) {$\vec{r}/q$};
\node [above] at (2,4) {$B(\vec{m}/q,\psi(q)/2q)$};
\node [above] at (1.5,5) {$B(\vec{m}/q,\psi(q)/q)$};
\draw [dashed] (1,4) -- (1,2);
\draw [dashed] (3,4) -- (3,2);
\draw [dashed] (1,2) -- (3,2);
\draw [dashed] (1,4) -- (3,4);
\draw [very thick] (.8,1) -- (4,4.2);
\draw [dashed] (4,1) -- (4,-2);
\draw [dashed] (.8,1) -- (.8,-2);
\draw [ultra thick] (.8,-2)--(3.2,-2);

\draw [thick, ->] (-2.5,-2) -- (0,-2) -- (5,-2);
\draw [thick, ->] (-1,-3) -- (-1,0) -- (-1,5);
\node [above] at (-1,5) {$\R^{d-n}$};
\node [right] at (5,-2) {$\R^n$};
\draw [dashed, shift={(2,3)}] (-2,-2) rectangle (2,2);
\draw[-), ultra thick] (3,-2) -- (3.2,-2);
\draw[(-, ultra thick] (.8,-2) -- (1,-2);

\end{tikzpicture}
\end{center}
\caption{Uniform neighborhoods}
\end{figure}

The goal is to reduce this problem to a ubiquitous system on $\R^n$ by considering a limsup set centered around these $\vec{p}/q$ which correspond to a sufficiently close $\vec{m}/q$. We would like to construct balls centered at the $\vec{p}/q \in \Q^n$ such that if $\vec{x} \in \R^n$ is inside the ball centered at $\vec{p}/q$, then $(\vec{x}, \tilde{\vec{x}}\tilde{\mathfrak{a}})$ is within $\psi(q)/q$ of the corresponding $\vec{m}/q$. However, with the current setup, there will not be balls centered at $\vec{p}/q\in \Q^n,$ but rather neighborhoods of varying sizes depending on the distance $\vec{m}/q$ is from $\Ln$, some of which may be to only one side of the rational point, as in Figure 1. 

To address this issue, consider a smaller class of rational points. Specifically only those $\vec{p}/q \in \Q^n \cap [0,1]^n$ such that
$$\lVert \hat{\vec{p}}\tilde{\mathfrak{a}} \rVert < \dfrac{\psi\left(|\hat{\vec{p}}|\right)}{2}.$$ Now, simply considering all $\vec{x} \in \R^n$ such that $(\vec{x},\tilde{\vec{x}}\tilde{\mathfrak{a}}) \in B(\vec{m}/q,\psi(q)/2q)$ would give a neighborhood which is not quite centered on $\vec{p}/q$. In order to rectify this, consider the ball centered at $\vec{p}/q$ of radius $\Psi(q)$ where
$$\Psi(q)=\dfrac{c\psi(q)}{q}$$
with $c=\left(2n|\mathfrak{a}|\right)^{-1}$, as in Figure 2.

Choosing this neighborhood allows us to more easily calculate the measure of the balls at different levels while being more conservative in determining which points are considered to be ``close enough."

\section{Counting problem}
In the course of the argument, we will need to establish an upper bound on the number of integer vectors $\vec{\hat{p}} \in \N^{n+1}$ defined above. Namely, on the size of the set of vectors
\begin{equation}\label{req}
 \left| \left\lbrace \hat{\vec{p}} \in \N^{n+1}: \lVert\hat{\vec{p}} \tilde{\mathfrak{a}} \rVert<\frac{\psi\left(k^{t}\right)}{2}, |\hat{\vec{p}}|\leq k^{t-1}, \vec{p}/q \in B\right\rbrace \right|  
\end{equation}
where $B=B(\vec{x}_0, \eta)$ so that $m(B)=(2\eta)^n$. Fix $q \in \N$ and define 
\begin{equation} \label{Pqdxe}
\mathcal{P}(q, \delta,\vec{x}_0, \eta)=|\{\vec{p} \in \N^n : \lVert \hat{\vec{p}}\tilde{\mathfrak{a}} \rVert< \delta, |\vec{p}-q \vec{x}_0| < q\eta \}|.
\end{equation}
In order to find a bound for (\ref{req}), we establish a bound on $\mathcal{P}(q,\delta,\vec{x}_0,\eta)$, for which the following theorem from \cite{huang2018} is needed. 
\begin{theorem}{\textup{\cite[Theorem 5]{huang2018}}} \label{jing}
Let $\mathfrak{a} \in \textbf{\textup{MAD}}(l,m,\omega).$ Then for all positive integers $J \geq 2$, we have 
\begin{equation}
\sum_{\substack{j \in \Z^m / \{\vec{0} \} \\ |\vec{j}| \leq J}}\prod_{u=1}^l \lVert \vec{j} \cdot \textup{row}_u(\mathfrak{a})\rVert^{-1} \ll J^{\omega}(\log J)^l 
\end{equation}
where the implied constant depends only on $\mathfrak{a}.$
\end{theorem}
The proof will also make use of a class of very useful trigonometric polynomials called the Selberg functions. A more detailed explanation of these functions can be found in \cite{montgomery1994ten}. Throughout this section let $e(y)=\exp(2 \pi i y).$
\begin{definition}Let $\Delta=(-\delta,\delta)\in \mathbb{T}$ and let $\chi_{\Delta}$ be its characteristic function. Then for any $J \geq 1$ there exist finite trigonometric polynomials of degree at most $J$
\begin{equation*}
S_J^{\pm}(y)=\sum_{|j| \leq J} b_j^{\pm}e(jy)
\end{equation*}
with
\begin{equation*}
|b_j^{\pm}| \leq \dfrac{1}{J+1}+\min \left(2 \delta, \dfrac{1}{\pi|j|}
\right)\end{equation*}
and
\begin{equation*}
b_0^{\pm}=2\delta \pm \dfrac{1}{J+1}
\end{equation*}
such that 
\begin{equation*}S_J^-(y) \leq \chi_{\Delta}(y) \leq S_J^+(y) \quad \forall y \in \mathbb{T}
\end{equation*}
These polynomials are called the Selberg functions.
\end{definition}
The following lemma will provide a bound on $\mathcal{P}(q, \delta,\vec{x}_0, \eta)$ when a certain Diophantine condition is imposed on the parametrizing matrix $\mathfrak{a}.$ The proof of this lemma follows closely the steps of the proof of  \cite[Theorem 7]{huang2018}.

\begin{lemma}\label{biglem}
If $q < \frac{1}{ 2\eta}$, then
\begin{equation}
\mathcal{P}(q,\delta,\vec{x}_0,\eta) \leq 1.
\end{equation}
Additionally, for $\mathfrak{a} \in \textbf{\textup{MAD}}(n,d-n,\omega)$ and $q \geq \frac{1}{2\eta},$ 
\begin{equation}
\mathcal{P}(q,\delta,\vec{x}_0,\eta) < 3^d\delta^{d-n}q^nm(B)+C_{\mathfrak{a}}\delta^{d-n-\omega}\log \left(\dfrac{1}{\delta}-1\right)^n.
\end{equation}
\end{lemma}
\begin{proof} Note that if $q< \frac{1}{2\eta}$ and there exists one $\vec{p}/q \in B$, then the distance to the nearest possible rational point is greater than the diameter of the ball. Thus, there can be at most one $\vec{p} \in \N^n$ satisfying the required conditions, so $\mathcal{P}(q, \delta, \vec{x}_0, \eta)\leq 1$ if $q < \frac{1}{2 \eta}.$ Therefore, restrict the focus to those $q \geq \frac{1}{2 \eta}.$
Notice that
\begin{align*}
\mathcal{P}(q,\delta,\vec{x}_0,\eta) &=\sum_{\vec{p} \in \Z^n \cap B(q\vec{x}_0,q\eta)} \prod_{v=1}^{d-n} \chi_{\Delta}(\hat{\vec{p}} \cdot \text{col}_v (\tilde{\mathfrak{a}})) \\
\intertext{and bounding the characteristic function above by $S_J^+$ this is} 
&<  \sum_{\vec{p} \in \Z^n \cap B(q\vec{x}_0,q\eta)} \prod_{v=1}^{d-n}S_J^+(\hat{\vec{p}} \cdot 
\text{col}_v (\tilde{\mathfrak{a}}))\\
\intertext{which can be rewritten using the properties of the Selberg functions as}
&=\sum_{\vec{p} \in \Z^n \cap B(q\vec{x}_0,q\eta)} \prod_{v=1}^{d-n}\sum_{|j_v| \leq J} b_{j_v}^+e(j_v\hat{\vec{p}} \cdot 
\text{col}_v (\tilde{\mathfrak{a}})). \\
\intertext{Reindexing, this is}
&=\sum_{\vec{p} \in \Z^n \cap B(q\vec{x}_0,q\eta)} \sum_{\substack{\vec{j} \in \Z^{d-n} \\|\vec{j}| \leq J|}} \left( \prod_{v=1}^{d-n}b^+_{j_v} \right) e\left( \sum_{v=1}^{d-n} j_v \hat{\vec{p}} \cdot \text{col}_v(\tilde{\mathfrak{a}}) \right) \\
\intertext{then pulling out the $\vec{j}=0$ term, and using the inequality}
 &|e(a_0x)+ e(a_1x) + \cdots + e(a_kx)| \leq \dfrac{2}{|e(x)-1|} \leq \dfrac{1}{\lVert x\rVert} \\
 \intertext{for any sequence of consecutive integers $a_i$, this is}
&\leq (2q\eta +1)^n\left|b_0^+ \right|^{d-n}+\sum_{0 < |\vec{j}|\leq J}  \left( \prod_{v=1}^{d-n}b^+_{j_v} \right) \left|\sum_{\vec{p} \in \Z^n \cap B(q\vec{x}_0,q\eta)} e\left(\hat{\vec{p}}\cdot\tilde{\mathfrak{a}}\cdot\vec{j}^T\right)\right|  \\
&\leq \left(2\delta+\dfrac{1}{J+1}\right)^{d-n}\left(\left(2q \eta+1\right))^n+\sum_{0 < |\vec{j}|\leq J} \prod_{u=1}^n \lVert \vec{j} \cdot \text{row}_u(\mathfrak{a} )\rVert^{-1} \right).\\
\intertext{Note that this step is where the reliance on $\vec{a}_0$ disappears. Then using Theorem \ref{jing} and the bound on $q$ this is}
&\leq  \left(2\delta+\dfrac{1}{J+1}\right)^{d-n}\left(2^nq^nm(B)+C_{\mathfrak{a}}\left(J^{\omega}\left(\log J\right)^n\right)\right),\\
\intertext{ and letting $J=\dfrac{1}{\delta}-1$ this is}
&\leq  \left(3\delta\right)^{d-n}\left(2^nq^nm(B)+C_{\mathfrak{a}}\left(\left(\dfrac{1}{\delta}-1\right)^{\omega}\log \left(\dfrac{1}{\delta}-1\right)^n\right)\right)  \\
&\leq  3^d\delta^{d-n}q^nm(B)+C_{\mathfrak{a}}\delta^{d-n-\omega}\log \left(\dfrac{1}{\delta}-1\right)^n 
\end{align*}
as was claimed.
\end{proof}

\begin{theorem}\label{mad theorems} Let $\mathfrak{a} \in \textup{Mat}_{n,d-n}$ such that for $0\leq s\leq n$, $\omega(\mathfrak{a})<\frac{n(d-n+s)}{n+1-s}$, let $\psi: \R \to \R^+$ be a function such that for all $\epsilon>0$, $\psi(q)\geq q^{\frac{s-n-1}{d-n+s}-\epsilon}$  for sufficiently large $q$, and let $B$ be an arbitrary ball in $[0,1]^n.$ Then there exists some $t_0 \gg 0$ which depends on $B, \mathfrak{a},$ and $\epsilon$ such that
\begin{equation}\label{theorem}
\# \left\lbrace \hat{\vec{p}} \in \N^{n+1}: \lVert\hat{\vec{p}} \tilde{\mathfrak{a}} \rVert<\dfrac{\psi\left(k^{t}\right)}{2}, |\hat{\vec{p}}|\leq k^{t-1}, \vec{p}/q \in B\right\rbrace   <3^{d+n+2}\psi(k^t)^{d-n}k^{(t-1)(n+1)}m(B)
\end{equation}
holds for all $t>t_0.$
\end{theorem}
\begin{remark} It will be shown later that this assumption on $\psi(q)$ can be made without loss of generality.
\end{remark}
\begin{proof} 

Note that since $\omega(\mathfrak{a})<\frac{n(d-n+s)}{n+1-s},$ then $\mathfrak{a} \in \textbf{MAD}(n,d-n,\omega)$ such that $\omega(\mathfrak{a})< \omega<\frac{n(d-n+s)}{n+1-s}$. Let $B=B(\vec{x}_0, \eta)$ such that $m(B)=(2 \eta)^n$ as above. Then 
\begin{align*}
& \# \left\lbrace \hat{\vec{p}} \in \N^{n+1}: \lVert\hat{\vec{p}} \tilde{\mathfrak{a}} \rVert<\dfrac{\psi\left(k^{t}\right)}{2}, |\hat{\vec{p}}|\leq k^{t-1}, \vec{p}/q \in B\right\rbrace  & \\
& = \sum_{q=1}^{k^{t-1}} \# \left\lbrace \vec{p} \in \N^n: \lVert \hat{\vec{p}}\tilde{\mathfrak{a}}\rVert< \dfrac{\psi\left(k^{t}\right)}{2}, \vec{p}/q \in B\right\rbrace  \\
& < \sum_{q=1}^{k^{t-1}} \# \left\lbrace \vec{p} \in \N^n: \lVert \hat{\vec{p}}\tilde{\mathfrak{a}}\rVert<\dfrac{\psi\left(k^{t}\right)}{2}, |\vec{p}-q\vec{x}_0| < q\eta\right\rbrace   \\
\intertext{which, by the definition of $\mathcal{P}(q, \delta,\vec{x}_0, \eta)$ is}
& =\sum_{q=1}^{k^{t-1}} \mathcal{P}(q,\psi(k^t)/2,\vec{x}_0, \eta) \\
&=\sum_{q \leq \frac{1}{2\eta}}\mathcal{P}(q,\psi(k^t)/2,\vec{x}_0, \eta)+\sum_{\frac{1}{2\eta} <q \leq k^{t-1}}\mathcal{P}(q,\psi(k^t)/2,\vec{x}_0, \eta) \\
\intertext{ which from Lemma \ref{biglem} is}
& <\sum_{q \leq \frac{1}{2\eta}}1+\sum_{\frac{1}{2\eta} <q \leq k^{t-1}}\left(3^d2^{n-d}\psi(k^t)^{d-n}q^nm(B ) +C_{\mathfrak{a}}2^{\omega-d+n}\psi(k^t)^{d-n-\omega}\log \left(\dfrac{1}{\psi(k^t)}-1\right)^n\right)\\
&<\dfrac{1}{2\eta}+3^d2^{n-d}\psi(k^t)^{d-n}k^{(t-1)(n+1)}m(B)+C_{\mathfrak{a}}2^{\omega-d+n}k^{t-1}\psi(k^t)^{d-n-\omega}\log \left(\dfrac{1}{\psi(k^t)}-1\right)^n \\
&\overset{\dagger}{<}3^{d}2^{n-d+1}\psi(k^t)^{d-n}k^{(t-1)(n+1)}m(B)+C_{\mathfrak{a}}2^{\omega-d+n}k^{t-1}\psi(k^t)^{d-n-\omega}\log \left(\dfrac{1}{\psi(k^t)}-1\right)^n \\
& \overset{\ddagger}{<}3^{d}2^{n-d+2}\psi(k^t)^{d-n}k^{(t-1)(n+1)}m(B) \\
&<3^{d+n+2}\psi(k^t)^{d-n}k^{(t-1)(n+1)}m(B)
\end{align*}
where the daggered inequality holds for $t \gg 0$ such that $k^{t(n^2+d)/n} >3^{-d}2^{n-d} k^{n+1}m(B)^{-(n+1)/n}$, and the double daggered inequality will hold if it can be shown that
\begin{equation}\label{error bound} C_{\mathfrak{a}}2^{\omega-d+n}k^{t-1}\psi(k^t)^{d-n-\omega}\log \left(\dfrac{1}{\psi(k^t)}-1\right)^n < 3^d2^{n-d+1}\psi(k^t)^{d-n}k^{(t-1)(n+1)}m(B).
\end{equation}
Therefore, inequality (\ref{theorem}) will follow from inequality (\ref{error bound}), or equivalently, from
\begin{align*}
C_{\mathfrak{a}}3^{-d}2^{\omega-1}\psi(k^t)^{-\omega}\log \left(\dfrac{1}{\psi(k^t)}-1\right)^n &< k^{(t-1)n}m(B) \\
\intertext{ If we use the assumption that $\psi(q) \geq q^{\frac{s-n-1}{d-n+s}-\epsilon}$, it suffices to show that}
C_{\mathfrak{a}}3^{-d}2^{\omega-1}k^{-t\omega\left(\frac{n(d-n+s)}{n+1-s}\right)}\log \left(k^{\frac{nt(d-n+s)}{n+1-s}}-1\right)^n &< k^{(t-1)n}m(B) \\
\intertext{which, rewritten, is}
C_{\mathfrak{a}}3^{-d}2^{\omega-1}m(B)^{-1}\log \left(k^{\frac{nt(d-n+s)}{n+1-s}}-1\right)^n &< k^{(t-1)n-t\omega\left(\frac{n(d-n+s)}{n+1-s}\right)} \\
C_{\mathfrak{a}}3^{-d}2^{\omega-1}k^{n}m(B)^{-1}\log \left(k^{\frac{nt(d-n+s)}{n+1-s}}-1\right)^n &< k^{t\frac{n(d-n+s)+\omega(s-n-1-\epsilon)}{d-n+s}}. \\
\end{align*}
Since the left hand side grows linearly and the right hand side grows exponentially so long as $n(d-n+s)+\omega(s-n-1)-\omega\epsilon>0$, the last inequality can be satisfied by taking a $t>0$ large enough. Since $\omega<\frac{n(d-n+s)}{n+1-s}$ and we can take $\epsilon>0$ as small as we like, inequality (\ref{error bound}) holds. Since inequality (\ref{error bound}) holds, we have shown inequality (\ref{theorem}).
\end{proof}

\section{Ubiquitous systems}

Theorems \ref{main theorem} and \ref{main jarnik} are proved using ubiquitous systems. The goal is to show that the rational points in Section \ref{index} together with a function $\rho$ and a constant $k$ form a \textit{local m-ubiquitous system}. A special case of a result from ubiquity theory is then used to obtain measure statements about a limsup set related to this subset of the rational points which will be helpful to our result. For a robust guide to this theory, the reader is referred to \cite{limsup}, and for a survey treatment, \cite{survey}.

The ubiquitous systems framework aims to generalize the concept of the classical limsup set $W(\psi)$ to limsup sets centered around a collection $\mathcal{R}=\{R_{\alpha}\}_{\alpha \in J}$ of \textit{resonant points}, $R_{\alpha}$, with  a function $\beta_{\alpha}: R_{\alpha} \mapsto \R^+$, the \textit{weight} of the resonant point. In the classical setup, $\mathcal{R}=\{\vec{p}/q \in \R^d\}$ and $\beta_{\vec{p}/q}:\vec{p}/q \mapsto q.$  The following definition gives a way to quantify whether or not there are ``enough" resonant points at all levels.

\begin{definition}[Local Ubiquity] Let $\mathcal{R}=\{R_{\alpha}\}_{\alpha \in J}$ be a collection of points in a metric space $\Omega$ and let $B=B(x,r)$ be an arbitrary ball with center $x \in \Omega$ and radius $r \leq r_0$. Suppose there exists a function $\rho$, and an absolute constant $\kappa>0$ such that for $t \geq t_0$,
\begin{equation}\label{ubiq} m\left(B \cap \bigcup_{k^{t-1}< \beta_{\alpha} \leq k^t} B\left(R_{\alpha},\rho(k^t)\right)\right) \geq \kappa m(B)
\end{equation}
where $t_0$ depends on $B$. Then the triple $(\mathcal{R},\rho, k)$ is said to be a \textit{local m-ubiquitous system}.
\end{definition}
Additionally, there is a limsup set which this framework can provide measure theoretic statements about. Centered at the resonant points, define the following set.
\begin{definition} For $\Psi: \R^+ \to \R^+$, let
\begin{equation*}
\Lambda(\Psi)=\limsup_{t \to \infty} \Delta(\Psi,t)
\end{equation*}
\end{definition}
where
\begin{equation*}
\Delta(\psi,t)=\bigcup_{\substack{\alpha \in J \\ k^{t-1} <\beta_{\alpha} \leq k^t}}B\left( R_{\alpha},\Psi(\beta_{\alpha})\right)
\end{equation*}
In the case where the resonant points are the rational points $\vec{p}/q \in \Q^n$ with weight $\beta: \vec{p}/q \mapsto q$, and with approximating function $\psi(q)/q$, this set is precisely the set of $\psi$-approximable numbers, $W(\psi)$.

The following theorem, a special case of \cite[Theorem 1, Corollary 2]{limsup}, and \cite[Theorem 2, Corollary 4]{limsup} says that showing local $m$-ubiquity along with a divergence condition is sufficient to obtain a Hausdorff measure result for $\Lambda(\Psi).$

\begin{theorem}\label{ubi}
\textup{\cite[Theorem 1, Corollary 4]{limsup}} Let $\Omega$ be a compact metric space with $\dim \Omega=n$ equipped with a non-atomic probability measure such that any open subset of $\Omega$ is $m$-measurable. Suppose that $(\mathcal{R},\rho,k)$ is a local $m$-ubiquitous system and that $\Psi$ is a regular approximating function. If
$$\sum_{t=1}^{\infty} \Psi(k^t)^s\rho(k^t)^{-n}=\infty,$$
for some $0 \leq s \leq n$, then 
$$\mathcal{H}^s(\Lambda(\Psi))=\mathcal{H}^s(\Omega).$$
\end{theorem}

In order to prove the main proposition of this section, a lemma is required which is a simple use of the following theorem from Minkowski.

\begin{theorem}\textup{(Minkowski's theorem for systems of linear forms)}\label{Minkowski}
Let $\beta_{i,j} \in \R$, where $1 \leq i,j \leq k$ and let $C_1, \ldots, C_k >0$. If
\begin{equation}
\left|\det(\beta_{i,j})_{1 \leq i, j \leq k}\right| \leq \prod_{i=1}^k C_i,
\end{equation}
then there exists a non-zero integer point $\vec{x}=(x_1, \ldots , x_k)$ such that
\begin{align}
\left|x_1 \beta_{i,1} + \cdots + x_k \beta_{i,k} \right|&<C_i \quad (1 \leq i \leq k-1) \\
\left|x_1 \beta_{k,1} + \cdots + x_k \beta_{k,k} \right|&\leq C_k
\end{align}
\end{theorem}
The proof of the above theorem is a simple use of Minkowski's convex body theorem, and can be found in \cite{survey}. The following lemma makes use of this theorem of Minkowski, and is necessary for the ubiquity approach.

\begin{lemma}\label{contain lemma} Let $B$ be an arbitrary ball in $[0,1]^n$. Then
\begin{equation}\label{contain}
B \subseteq \bigcup_{\substack{\hat{\vec{p}} \in \N^{n+1}\\ |\hat{\vec{p}}| \leq N\\\lVert\hat{\vec{p}} \tilde{\mathfrak{a}} \rVert<\psi(N)/2}}B\left(\vec{p}/q,\dfrac{2^{(d-n)/n}}{N^{(n+1)/n}\psi(N)^{(d-n)/n}} \right).
\end{equation}
\end{lemma}

\begin{proof}
For any $\vec{x}=(x_1, \ldots, x_n) \in B$,
\begin{equation*}
\vec{x} \in  \bigcup_{\substack{\hat{\vec{p}} \in \N^{n+1}\\ |\hat{\vec{p}}|\leq N\\\lVert\hat{\vec{p}} \tilde{\mathfrak{a}} \rVert<\psi(N)/2}}B\left(\vec{p}/q,\dfrac{2^{(d-n)/n}}{N^{(n+1)/n}\psi(N)^{(d-n)/n}} \right)
\end{equation*}
if the following system of linear forms has an integer solution.

\begin{align*}
-r_1+p_1a_{2,1}+ \cdots +p_na_{n+1,n}+qa_{1,1} &< \psi(N)/2 \\
 \vdots \\
-r_{d-n}+p_1a_{2,d-n}+ \cdots +p_na_{n+1,d-n}+qa_{1,d-n} &< \psi(N)/2 \\
-p_1+qx_1 &< \dfrac{2^{(d-n)/n}}{N^{1/n}\psi(N)^{(d-n)/n}}  \\
\vdots \\
-p_n+qx_1 &< \dfrac{2^{(d-n)/n}}{N^{1/n}\psi(N)^{(d-n)/n}}  \\
q  &\leq N
\end{align*}
where $r_1, \ldots r_{d-n}\in \Z$, $\vec{p}={p_1, \ldots, p_n}$, and $a_{i,j}$ are the entries of $\tilde{\mathfrak{a}}$. Then the desired integer point is $(r_1, \ldots, r_{d-n},p_1, \ldots, p_n, q) \in \Z^{d+1}$ and the matrix $(\beta_{i,j})$ is equal to the block matrix
\begin{center}
\begin{equation}
\left(\begin{array}{c|c}
-I_{d-n} & \tilde{\mathfrak{a}}^T  \\
\hline 
0 & A 
\end{array}\right)
\end{equation}
\end{center}
where $I_{d-n}$ is the $d-n$ identity matrix, $0$ is the $d-n,n+1$ zero matrix, and $A$ is the block matrix defined by
\begin{center}
\begin{equation}
A=\left(\begin{array}{c|c}
-I_{n} & \vec{x}^T  \\
\hline 
0 & 1 
\end{array}\right)
\end{equation}
\end{center}
where $I_{n}$ is the $n$ identity matrix and $0$ is the zero $n$-vector. Now $\left|\det(\beta_{i,j})_{1 \leq i, j \leq k}\right|=1$ and 
\begin{equation*}
\prod_{i=1}^k C_i= \left( \dfrac{\psi(N)}{2}\right)^{d-n} \left(\dfrac{2^{(d-n)/n}}{N^{1/n}\psi(N)^{(d-n)/n}}\right)^n N=1
\end{equation*}
so 
\begin{equation*}
\left|\det(\beta_{i,j})_{1 \leq i, j \leq k}\right| \leq \prod_{i=1}^k C_i
\end{equation*}
and by Theorem \ref{Minkowski}, there is an integer solution to the above system of linear forms, hence (\ref{contain}) holds.
\end{proof}

For the purpose of this problem, define 
$$\mathcal{J}=\{\hat{\vec{p}}=(q,\vec{p}) \in \N \times \Z^n : \lVert\hat{\vec{p}} \tilde{\mathfrak{a}} \rVert< \psi \left( \left| \hat{\vec{p}}\right| \right)/2 \}$$
$$\mathcal{R}=\{\vec{p}/q \in \Q^n \cap [0,1]^n: \hat{\vec{p}} \in \mathcal{J}\}$$
$$\beta_{\hat{\vec{p}}}=|\hat{\vec{p}}|,$$
and 
$$\rho(q)=\dfrac{2^{(d-n)/n}}{q^{(n+1)/n}\psi(q)^{(d-n)/n}}.$$
With these definitions, we prove the following proposition.

\begin{proposition}\label{prop} For $\tilde{\mathfrak{a}} \in \textup{Mat}_{n+1,d-n}$ and $0 \leq s \leq n$ such that $\omega(\mathfrak{a})<\frac{n(d-n+s)}{n+1-s}$, and with a function $\psi: \R \to \R^+$ such that for all $\epsilon>0$, $\psi(q) \geq q^{\frac{s-n-1}{d-n+s}-\epsilon}$ for sufficiently large $q$, the triple $(\mathcal{R},\rho,k)$ with $k >2^{(n+d)/(n+1)}3^{(n+d+2)/(n+1)}$ forms a local $m$-ubiquitous system.

\end{proposition}
\begin{proof}

Let $B$ be an arbitrary ball in $[0,1]^n$. By Lemma \ref{contain lemma},
\begin{equation} 
m\left(B \cap \bigcup_{\substack{\hat{\vec{p}} \in \N^{n+1}\\ |\hat{\vec{p}}| \leq N\\\lVert\hat{\vec{p}} \tilde{\mathfrak{a}} \rVert<\psi(N)/2}}B\left(\vec{p}/q,\dfrac{2^{(d-n)/n}}{N^{(n+1)/n}\psi(N)^{(d-n)/n}} \right)\right) =m(B),
\end{equation}
therefore, instead of directly showing that there is some $0 < \kappa <1$ for which
\begin{equation}\label{ubi1}
m\left(B \cap \bigcup_{\substack{\hat{\vec{p}} \in \N^{n+1}\\ k^{t-1} <|\hat{\vec{p}}| \leq k^{t}\\\lVert\hat{\vec{p}} \tilde{\mathfrak{a}} \rVert<\psi(k^{t})/2}}B\left(\vec{p}/q,\dfrac{2^{(d-n)/n}}{q^{(n+1)/n}\psi(q)^{(d-n)/n}} \right)\right) \geq \kappa m(B)
\end{equation}
holds, we will instead show that 

\begin{equation} m\left(B \cap \bigcup_{\substack{\hat{\vec{p}} \in \N^{n+1}\\ |\hat{\vec{p}}| \leq k^{t-1}\\ \lVert\hat{\vec{p}} \tilde{\mathfrak{a}}\rVert<\psi(k^{t})/2}}B\left(\vec{p}/q,\dfrac{2^{(d-n)/n}}{k^{(n+1)t/n}\psi(k^t)^{(d-n)/n}} \right)\right) < m(B)(1-\kappa)
\end{equation}
from which it follows that \eqref{ubi1} holds.

Note that
\begin{align*}
&m\left(B \cap \bigcup_{\substack{\hat{\vec{p}} \in \N^{n+1}\\ |\hat{\vec{p}}| \leq k^{t-1}\\ \lVert\hat{\vec{p}} \tilde{\mathfrak{a}} \rVert<\psi(k^{t})/2}}B\left(\vec{p}/q,\dfrac{2^{(d-n)/n}}{k^{(n+1)t/n}\psi(k^t)^{(d-n)/n}} \right)\right)  \\
&< \sum_{\substack{\hat{\vec{p}} \in \N^{n+1}\\ \vec{p}/q \in 2B \\ |\hat{\vec{p}}| \leq k^{t-1}\\ \lVert\hat{\vec{p}} \tilde{\mathfrak{a}} \rVert<\psi(k^{t})/2}}m\left(B\left(\vec{p}/q,\dfrac{2^{(d-n)/n}}{k^{(n+1)t/n}\psi(k^t)^{(d-n)/n}} \right)\right) \\
\intertext{ where $2B$ is a doubling of $B$ to account for balls centered at points $\vec{p}/q \not\in B$ which intersect $B$. Note that a doubling is large enough to capture all $\vec{p}/q$ which are the centers of balls which could intersect $B$ if $t >t_0$ where $t_0  \gg 0$ depends only on $\mathfrak{a}, B,$ and $ \epsilon$. This is then}
 &=\dfrac{2^d}{k^{(n+1)t}\psi(k^t)^{d-n}}\# \left\lbrace \hat{\vec{p}} \in \N^{n+1}: \lVert\hat{\vec{p}} \tilde{\mathfrak{a}} \rVert<\dfrac{\psi\left(k^{t}\right)}{2}, |\hat{\vec{p}}|\leq k^{t-1}, \vec{p}/q \in 2B\right\rbrace \\ 
\text{which, by Theorem \ref{mad theorems} is} \\
&< \dfrac{2^d}{k^{(n+1)t}\psi(k^t)^{d-n}}\left(3^{n+d+2} \psi(k^t)^{d-n}k^{(t-1)(n+1)}m(2B)\right) \\
& =\dfrac{2^{n+d}3^{n+d+2}}{k^{n+1}}m(B) \\
\intertext{taking $k>2^{(n+d)/(n+1)}3^{(n+d+2)/(n+1)}$, this is}
&<(1-\kappa) m(B)
\end{align*}
for some $0<\kappa<1.$
\end{proof}

We are now ready for the proof of the main theorems.

\section{Proof of main theorems}

\begin{proof}

For  $\psi: \R \to \R^+$ decreasing such that for all $\epsilon>0$, $\psi(q) \geq q^{\frac{s-n-1}{d-n+s}-\epsilon}$ for sufficiently large $q$, let 
$$\mathcal{J}=\{\hat{\vec{p}}=(q,\vec{p}) \in \N \times \Z^n : \lVert\hat{\vec{p}} \tilde{\mathfrak{a}} \rVert< \psi \left( \left| \hat{\vec{p}}\right| \right)/2 \}$$
$$\mathcal{R}=\{\vec{p}/q \in \Q^n \cap [0,1]^n: \hat{\vec{p}} \in \mathcal{J}\}$$
$$\beta_{\hat{\vec{p}}}=|\hat{\vec{p}}|,$$
and 
$$\rho(q)=\dfrac{2^{(d-n)/n}}{q^{(n+1)/n}\psi(q)^{(d-n)/n}}.$$
By Proposition \ref{prop}, $(\mathcal{R},\rho,k)$ forms a local $m$-ubiquitous system. This restriction on $\psi(q)$ will soon be show to be taken without loss of generality. Let
$$\Psi(q)=\dfrac{c\psi(q)}{q}$$
with
$$c=\left(2n|\mathfrak{a}|\right)^{-1}. $$
By the definition of $\Psi(q)$, we have
\begin{equation*}
\dfrac{\Psi(k^{t+1})}{\Psi(k^t)}=\dfrac{c\psi(k^{t+1})}{ck\psi(k^t)}	\leq \dfrac{1}{k},
\end{equation*}
so $\Psi$ is a regular function. Also, we have that
\begin{equation*}
\sum_{t \in \N} \Psi(k^t)^s\rho(k^t)^{-n}  =\left(\dfrac{c^s}{2^{(d-n)}}\right)\sum_{t \in \N} k^t\psi(k^t)^{d-n+s}k^{t(n-s)}
\end{equation*}
which diverges by Cauchy's condensation test, so by Theorem \ref{ubi}, $\mathcal{H}^s(\Lambda(\Psi) \cap [0,1]^n)=\mathcal{H}^s([0,1]^n)$. The following lemma translates this result back to $\Ln$.

\begin{lemma}\label{ubisubspace} If  $\mathcal{H}^s(\Lambda(\Psi) \cap [0,1]^n)=\mathcal{H}^s([0,1]^n)$ then $\mathcal{H}^s(W(\psi) \cap (\Ln \cap \mathcal{S}_{\vec{0}}))=\mathcal{H}^s([0,1]^n)$.
\end{lemma}

\begin{proof} The mapping $f:[0,1]^n \to \Ln \cap \mathcal{S}_{\vec{0}}$ given by $\vec{x} \mapsto (\vec{x}, \tilde{\vec{x}}\tilde{\mathfrak{a}})$ is one-to-one, thus it only remains to be shown thatif $\vec{x} \in \Lambda(\psi)$, then $f(\vec{x}) \in W(\psi).$ We will show that if $\vec{x} \in B(\vec{p}/q, \Psi(q)),$ for $\vec{p}/q \in \Q^d \cap [0,1]^n$ such that $\hat{\vec{p}} \in \mathcal{J}$ then $(\vec{x}, \tilde{\vec{x}}\tilde{\mathfrak{a}}) \in B_{\mathcal{A}}(\vec{m}/q,\psi(q)/q)$ for some $\vec{m}\in \Z^d.$

For any $\vec{p}/q \in \Q^d \cap [0,1]^n$ such that $\hat{\vec{p}} \in \mathcal{J}$, $\lVert \hat{\vec{p}}\tilde{\mathfrak{a}} \rVert < \psi(|\hat{\vec{p}}|)/2$ by definition of the index set $\mathcal{J}$, so there exists some $\vec{m} \in \R^d$ such that $(\vec{p}/q, \widetilde{\vec{p}/q}\tilde{\mathfrak{a}}) \in B(\vec{m}/q, \psi(q)/2q)$. Therefore, $|\vec{m}/q-(\vec{p}/q,\widetilde{\vec{p}/q}\tilde{\mathfrak{a}})|<\psi(q)/2q$. Since $\vec{x} \in B(R_{\hat{\vec{p}}}, \Psi(q))$, $|\vec{p}/q-\vec{x}|< \Psi(q)<\psi(q)/2q$ and then

\begin{align*}
|\widetilde{\vec{p}/q}\tilde{\mathfrak{a}}-\tilde{\vec{x}}\tilde{\mathfrak{a}}|&=
|(\vec{p}/q-\vec{x})\mathfrak{a}| \\
&=\max_{v=1,\ldots, d-n}\left(\vec{p}/q-\vec{x} \right) \text{col}_v(\mathfrak{a})\\
&<\max_{v=1,\ldots, d-n} \left\lbrace \left(\dfrac{c\psi(q)}{q}\right)a_{1,v}+\cdots +\left(\dfrac{c\psi(q)}{q}\right)a_{n,v}\right\rbrace \\
\intertext{ where $a_{i,j}$ are the entries of $\mathfrak{a}$. This is then}
&<\max_{v=1,\ldots, d-n} \left\lbrace \left(\dfrac{c\psi(q)}{q}\right)|\text{col}_v(\mathfrak{a})|+\cdots +\left(\dfrac{c\psi(q)}{q}\right)|\text{col}_v(\mathfrak{a})|\right\rbrace \\
&<n\dfrac{c\psi(q)}{q}|\mathfrak{a}|\\
\intertext{which, by the choice of $c$, is}
&<\dfrac{\psi(q)}{2q}.
\end{align*} 
Therefore $|(\vec{p}/q, \widetilde{\vec{p}/q}\tilde{\mathfrak{a}})-(\vec{x}, \tilde{\vec{x}}\tilde{\mathfrak{a}})|<\psi(q)/2q.$   Then by the triangle inequality, $|\vec{m}/q-(\vec{x}, \tilde{\vec{x}}\tilde{\mathfrak{a}})|<\psi(q)/q$ so $(\vec{x}, \tilde{\vec{x}}\tilde{\mathfrak{a}}) \in B(\vec{m}/q), \psi(q)/q).$
\end{proof}

By the above lemma, $\mathcal{H}^s(W(\psi) \cap (\Ln \cap \mathcal{S}_{\vec{0}}))=\mathcal{H}^s([0,1]^n)$, but it still remains to be shown that $\mathcal{H}^s(W(\psi) \cap \Ln)=\mathcal{H}^s(\Ln)$.

Since $W(\psi)$ is invariant under integer shifts, any segment $\Ln\cap\mathcal{S}_{\vec{v}}$ of $\Ln$ lying in another strip of $\R^d$ can be translated into $\mathcal{S}_{\vec{0}}$ and considered as a segment there. This shift will affect $\vec{a}_0$, so that this segment is defined in $\mathcal{S}_{\vec{0}}$ by $\tilde{\mathfrak{a}}_{\vec{v}}=\left[\begin{array}{c}\vec{a}_0+\vec{v}\mathfrak{a} \\ \mathfrak{a} \end{array}\right]$. Since both $\tilde{\mathfrak{a}}$ and $\tilde{\mathfrak{a}}_{\vec{v}}$ have $\mathfrak{a}$ as the matrix parametrizing the tilt, the above argument shows that $\mathcal{H}^s(\mathcal{L}_{\tilde{\mathfrak{a}}_{\vec{v}}}\cap \mathcal{S}_{\vec{0}})=\mathcal{H}^s(\Ln\cap [0,1]^n)$ where $\mathcal{L}_{\tilde{\mathfrak{a}}_{\vec{v}}}$ is the affine subspace parametrized by $\tilde{\mathfrak{a}}_{\vec{v}}.$ Since $\mathcal{H}^s(\mathcal{L}_{\tilde{\mathfrak{a}}_{\vec{v}}}\cap \mathcal{S}_{\vec{0}})=\mathcal{H}^s(\Ln \cap \mathcal{S}_{\vec{v}})$ and $\cup_{\vec{v} \in \Z^n} \mathcal{S}_{\vec{v}}=\R^d$,  $\mathcal{H}^s(W(\psi) \cap \Ln)=\mathcal{H}^s(\Ln)$. 

All that is left to show is that we can make the assumption that $\psi(q) \geq q^{\frac{s-n-1}{d-n+s}-\epsilon}$ used in Theorem \ref{mad theorems} without loss of generality, for which we will need the following lemma, the proof of which can be found in \cite[Lemma 3.10]{dod}.

\begin{lemma}\textup{(Hausdorff-Cantelli Lemma)} 
Let $E$ be a set in $\R^n$ and suppose that 
$$E \subseteq \{t \in \R^n : t \in H_j \text{ for infinitely many } j \in \N \}$$
where $\{H_j\}_{j \in \N}$ is a family of hypercubes. If for some $s >0,$
$$\sum_{j=1}^{\infty} \ell(H_j)^s < \infty,$$
then $\mathcal{H}^s(E)=0$.
\end{lemma}

 Let $\phi=q^{\frac{s-n-1}{d-n+s}-\epsilon}$, and define $\bar{\psi}(q)=\max\{\psi(q),\phi(q)\}$. Then $\bar{\psi}$ satisfies all our assumptions, so almost every point on $\Ln$ is $\bar{\psi}$-approximable.

In $[0,1]^n$, consider the set 
\begin{equation}
W_t(\phi)=\bigcup_{\substack{\hat{\vec{p}} \in \N^{n+1}\\ |\hat{\vec{p}}|<k^{t-1} \\ \lVert\hat{\vec{p}} \tilde{\mathfrak{a}} \rVert<\phi(k^t)}} B\left( \vec{p}/q, \dfrac{\phi(k^t)}{k^t}\right) \cap [0,1]^n
\end{equation}
and note that $\limsup_{t \to \infty} W_t(\phi)=\Lambda(\Psi)\cap[0,1]^n$ where $\Lambda (q)=\phi(q)/q.$ Then 
\begin{align*}
\sum_{q \in \N} \ell(W_t(\phi))^s &\leq \sum_{t \in \N}\left( \dfrac{2\phi(k^t)}{k^t}\right)^s \#\{\hat{\vec{p}} \in \N^{n+1} : \lVert\hat{\vec{p}}\tilde{\mathfrak{a}}\rVert<\phi(k^t), |\hat{\vec{p}}|<k^{t-1}\} \\
& \leq \dfrac{2^s3^{d+n+2}}{k^{n+1}}\sum_{t \in \N} k^t\phi(k^t)^{d-n+s}k^{t(n-s)} \\
\end{align*}
which converges because $\sum_{q \in \N} \phi(q)^{d-n+s}q^{n-s}=\sum_{q \in \N}q^{-1-\epsilon}$ does by the Cauchy condensation test. Thus $\mathcal{H}^s(\limsup(W_t(\phi)))=\mathcal{H}^s(\Lambda(\Psi) \cap [0,1]^n)=0$ by the Hausdorff-Cantelli lemma, and by the same argument as above, $\mathcal{H}^s(W(\phi)\cap \Ln)=0.$ Since every point which is $\bar{\psi}$-approximable but not $\phi$-approximable is $\psi$-approximable, $\mathcal{H}^s(W(\psi)\cap \Ln)=\mathcal{H}^s(\Ln).$
\end{proof}

\begin{acknowledgments} \textup{The author would like to thank his advisor, Felipe Ram\'irez, for his advisement and mentorship throughout the course of this project. He would also like to thank Jing-Jing Huang for the helpful correspondence.}

\end{acknowledgments}

\bibliography{bib}

\end{document}